\documentclass[12pt]{amsart}
\usepackage{amssymb}
\usepackage{amsmath,amssymb,amsfonts,amsthm,graphics,
latexsym, amscd, amsfonts, epsfig, eepic,epic,psfrag,setspace}
\usepackage{mathrsfs}
\usepackage{color}
\usepackage{eucal}
\usepackage{xspace}
\usepackage{amsthm}
\usepackage{mathrsfs}

\theoremstyle{plain}
\newtheorem{theorem}{Theorem}

\newtheorem{lemma}{Lemma}

\theoremstyle{definition}
\newtheorem{definition}{Definition}

\theoremstyle{remark}

\numberwithin{equation}{section}

\begin{document}

\doublespacing

\begin{center}
{\bf\Large A phase transition in energy-filtered RNA secondary structures}
\\
\vspace{15pt} Hillary S.~W. Han and Christian M. Reidys$^{\,\star}$
\end{center}

\begin{center}
         Department of Mathematics and Computer Science  \\
         University of Southern Denmark, Denmark \\
         Phone: *45-24409251 \\
         Fax: *45-65502325 \\
         email: duck@santafe.edu
\end{center}

\centerline{\bf Abstract}
In this paper we study the effect of energy parameters on minimum free
energy (mfe) RNA secondary structures. Employing a simplified combinatorial
energy model, that is only dependent on the diagram representation and that
is not sequence specific, we prove the following dichotomy result.
Mfe structures derived via the Turner energy parameters contain only
finitely many complex irreducible substructures and just minor parameter
changes produce a class of mfe-structures that contain a large number of
small irreducibles.
We localize the exact point where the distribution of irreducibles
experiences this phase transition from a discrete limit to a central limit
distribution and subsequently put our result into the context of quantifying
the effect of sparsification of the folding of these respective
mfe-structures. We show that the sparsification of realistic mfe-structures
leads to a constant time and space reduction and that the sparsifcation of
the folding of structures with modified parameters leads to a linear time
and space reduction.
We furthermore identify the limit distribution at the phase transition
as a Rayleigh distribution.

{\bf Keywords}: RNA secondary structure, loops-based, energy model,
                dominant singularity, limit distribution


\section{Introduction}

An RNA molecule is described by its primary structure, a linear string
composed of the nucleotides {\bf A}, {\bf G}, {\bf U} and {\bf C}, referred
to as the backbone. Each nucleotide can form a base pair by interacting
with at most one other nucleotide by establishing hydrogen bonds. Here
we restrict ourselves to
Watson-Crick base pairs {\bf GC} and {\bf AU} as well as the wobble base
pairs {\bf GU}. In the following, base triples as well as other types of more
complex interactions are neglected. RNA structures can be presented as
diagrams by drawing the backbone horizontally and all base pairs as arcs
in the upper halfplane; see Figure~\ref{F:RNAp}.
This set of arcs is tantamount to our notion of coarse-grained RNA
structure. In particular, we shall ignore any spatial embedding or
geometry of the molecule beyond its collection of base pairs.
\begin{figure}[ht]
\centerline{\epsfig{file=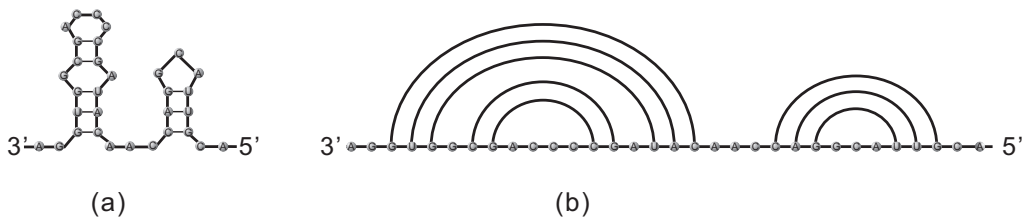,width=0.9\textwidth} \hskip8pt}
\caption{\small (A) An RNA secondary structure and (B) its diagram
representation.
}\label{F:RNAp}
\end{figure}
Accordingly, particular classes of base pairs translate into specific
structure categories, the most prominent of which are secondary
structures \cite{Kleitman:70,Nussinov:78,waterman:78a,waterman:79}.
When represented as diagrams, secondary structures have only non-crossing
base pairs (arcs). In the following an RNA secondary structure is
tantamount to a diagram without any crossing arcs. The combinatorics and
prediction of RNA secondary from primary structure was pioneered three
decades ago by Michael Waterman \cite{waterman:79, waterman:78a, waterman:80, waterman:94}.

RNA structures are a result of a folding of the primary sequence. The folded
configurations are energetically somewhat optimal. Here energy means
free energy, which is dominated by the loops forming between adjacent base
pairs and not by the hydrogen bonds of the individual base pairs
\cite{Mathews:99}. In addition sterical constraints imply certain minimum
arc-length conditions for minimum free energy configurations
\cite{Waterman:78aa}.

For a given RNA sequence polynomial-time dynamic programming (DP) algorithms
can be devised, finding such minimal energy configurations. The most commonly
used tools predicting simple RNA secondary structure \texttt{mfold}
\cite{Zuker:89} and the \texttt{Vienna RNA Package} \cite{Hofacker:94a},
are running at $O(N^2)$ space and $O(N^3)$ time solution.

In the context of polynomial-time DP algorithms a particular method, the
sparsification has been devised. Sparsification is tailored to speed up
DP-algorithms predicting minimum free energy (mfe)-secondary structures
\cite{spar:07,Backofen:11} by pruning certain computation paths
encountered in the DP-recursions, see Fig.~\ref{F:spar_sec}.
\begin{figure}[ht]
\begin{center}
\includegraphics[width=0.55\columnwidth]{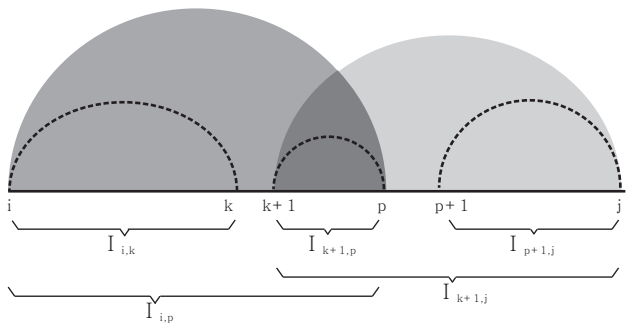}
\end{center}
\caption{\small Sparsification: suppose the optimal solution $I_{i,j}$ is
obtained from the optimal solutions $I_{i,k}$, $I_{k+1,p}$ and $I_{p+1,j}$.
Based on the recursions of the secondary
structures, $I_{i,k}$ and $I_{k+1,p}$ produce an optimal solution of
$I_{i,p}$. Similarly, $I_{k+1,p}$ and $I_{p+1,j}$ produce an optimal solution
of $I_{k+1,j}$.
Now, in order to obtain an optimal solution of $I_{i,j}$ it is sufficient to
consider either the grouping $I_{i,p}$ and $I_{p+1,j}$ or $I_{i,k}$ and
$I_{k+1,j}$.}
\label{F:spar_sec}
\end{figure}
In the context of the folding of RNA secondary structures, sparsification
reduces the DP-recursion paths to be based on so called candidates. A
candidate is in this case an interval, for which the optimal solution
cannot be written as a sum of optimal solutions of sub-intervals. Tracing
back these candidates gives rise to irreducible structures and the crucial
observation is here that these irreducibles appear only at a low frequency.
This means that there are only relatively few candidates have to be
considered, which in turn implies a significant reduction in time and space
complexity.

In this paper we study the effect of the particular choice of energy
parameters on sparsification. In other words, we study what happens to
RNA structures of a certain energy and quantify the effect of
sparsification of their mfe-folding. We shall see that this energy filtration
of secondary structures can have a dramatic effect on the structures having
significant implication for the effect of sparsification.

The energy parameters are associated to the loops of a secondary
structure, they are empirically measured enthalpic and entropic terms that
depend on loop sequence, length and type \cite{SantaLucia:96, Mathews:99}.
We shall restrict ourselves to a simplified notion of energy that does
not take into account the specifics of nucleotides but only depends on
the combinatorial representation the secondary structure. For instance,
the free energy of a hairpin loop $H$, $G_H$, is given by
\begin{equation}
G_{H} = \alpha_1 + \alpha_2\,\ell_{H}+ \alpha_3 \quad \text{\rm and} \quad
Q_{H} = v^{G_{H}},
\end{equation}
$\alpha_1$ being the penalty for forming $H$, $\alpha_2$ the penalty
for an unpaired base, $\alpha_3$ the score associated to a tetra-loop,
$\ell_{H}$ denoting the number of unpaired bases and $Q_{H}$ being the
weight of $H$. The other two loop-types are treated along these lines.

Equipped with this notion of ``combinatorial'' energy we study
the energy filtration of RNA secondary structures. In light of
the above discussion about the candidates and sparsification
we pay particular attention to irreducible RNA secondary
structures. One key question here being: under which conditions
does the probability of finding an irreducible minimum-free
energy structure tend to zero for larger and larger sequences?

The main results of this paper are as follows: we will show that
\begin{itemize}
\item for energy-parameters mimicking the established Turner energy model
      \cite{Mathews:99} sparsification implies a time and space reduction
      by a constant factor,
\item there exist energy-parameters close to those of the Turner energy
      model \cite{Mathews:99} for which sparsification implies a linear
      time reduction,
\item the effect of sparsification is closely connected to the distribution
      of irreducibles within secondary structures. To be precise we prove
      that this distribution undergoes a phase transition from a discrete
      limit law to a central limit law,
\item the limit distribution of irreducibles at the phase transition is
      a Rayleigh-law and a DP-folding in this regimen experiences a
      linear reduction in space and time.
\end{itemize}

\section{Some basic facts}

As mentioned above, we present an RNA secondary structure as a diagram
by drawing its backbone as a horizontal line containing vertices
corresponding to the labels of the nucleotides and each Watson-Crick
base pair as an arc or chord in the upper halfplane. Consequently, the
diagram representation ignores the particular type of nucleotides. Any
such pair may be inserted in two positions $i,j$ incident to an arc as
long as it is compatible with the Watson-Crick and {\bf G-U} base paring
rules.

The length of one chord $(i,j)$ is $\ell=j-i$ and in this paper, we only
consider RNA secondary structure with chord length $\geq 4$, see
Fig. ~\ref{F:secondary-struct}.
A chord connecting the first and last vertices is called a rainbow and
a RNA secondary structures exhibiting a rainbow is an irreducible RNA
secondary structure. This notion of irreducibility comes up naturally
when one decomposes a structure by means of cutting the backbone in two
positions without breaking any arcs.
\begin{figure}[ht]
\centerline{%
\epsfig{file=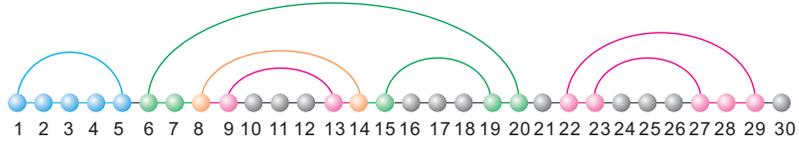,width=0.7\textwidth}\hskip15pt }
\caption{\small A diagram representation of secondary structure with
$\geq 4$.}
\label{F:secondary-struct}
\end{figure}
Let $\mathscr{S}(n)$ and $\mathscr{C}(n)$ denote the collections of
all linear chord diagrams of RNA secondary structures and irreducible
secondary structures on $n$ vertices, respectively. Let ${\bf s}(n)$ and
${\bf c}(n)$ denote the cardinalities of these sets.
Furthermore, setting
$\mathcal{S}=\cup_{n \geq 0} \mathscr{S}(n)$ and
$\mathcal{C}=\cup_{n \geq 0} \mathscr{C}(n)$ denote the set of
all linear chord diagrams of RNA secondary structures and irreducible
secondary structures, let $s \in \mathcal{S}$ and  $c \in \mathcal{C}$
denote a $\mathcal{S}$- or $\mathcal{C}$-structure.

Let $\mathscr{S}(n, j, i) \supseteq \mathscr{C}(n, j, i)$ denote
the collections of RNA secondary structures and irreducible RNA
secondary structures on $n \geq 0$ vertices having length $n$, energy $j$ and
weight $i$, then
\begin{equation}
\begin{split}
{\bf S}(z,v,p)=
\sum_{n \geq 0} \sum_{j}\, \sum_{i \geq 0} {\bf s}(n, j, i) z^n v^j p^i\\
{\bf C}(z,v,p)=
\sum_{n \geq 0} \sum_{j}\, \sum_{i \geq 0} {\bf c}(n, j, i) z^n v^j p^i.
\end{split}
\end{equation}

Thermodynamic models for nucleic acid secondary structure are based on a
decomposition of the base-pairing diagram of structures into distinct loops
that are associated with empirically measured enthalpic and entropic terms
that depend on loop sequence, length and type \cite{SantaLucia:96, Mathews:99}.
To obtain a better idea about these loops we give the diagram representations
of three types of Loops $L$:

\begin{itemize}
\item a {\sf Hairpin-loop} ($H$) consists of a chord $(i,j)$ with a sequence
of unpaired bases $[i+1, j-1]$. In particular, we have the restriction
$j-i \geq 4$; if $j-i=4$ ({\sf tetra-loop});
\item an {\sf interior-loop} ($I$) consists of two base pairs $(i,j)$ and
$(i_1, j_1)$ and three sequence of unpaired bases
$[i+1, i_1-1]$, $[j_1+1, j-1]$ and $[i_1-1, j_1-1]$. The case of
$i+1=i_1-1$ and $i_1-1=j_1-1$ is referred to as {\sf helix};
\item a {\sf multi-loop} ($M$) is a sequence:
$$( (i,j), (i_1, j_1), \ldots, (i_k, j_k), [i, i_1],
[j_{k-1}+1, i_{k}-1], [j_k+1, j])
$$
with sequences of unpaired bases $[j_{k-1}+1, i_{k}-1]$, for $k \geq 2$,
see Fig.~\ref{F:loops}.
\end{itemize}

\begin{figure}[ht]
\centerline{%
\epsfig{file=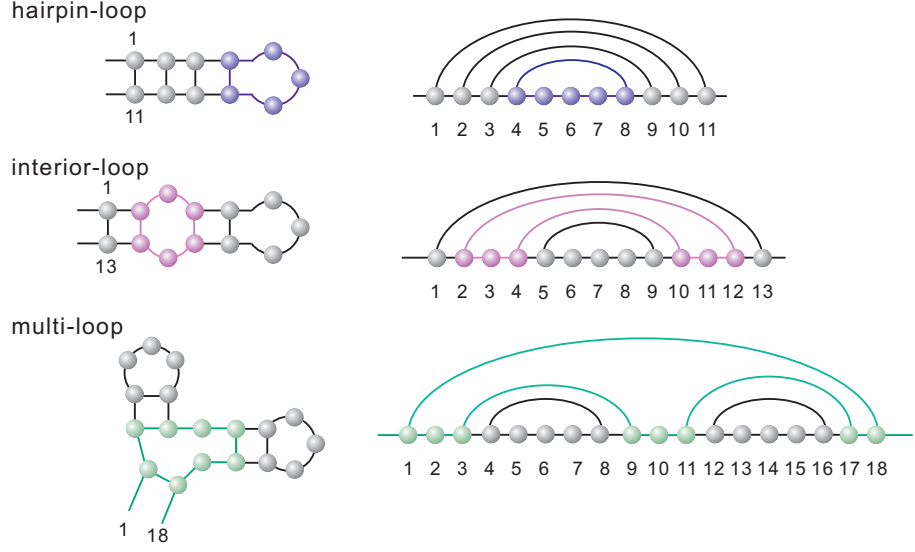,width=0.8\textwidth}\hskip15pt }
\caption{\small It shows hairpin-loop $H$, interior-loop $I$ and
multi-loop $M$.} \label{F:loops}
\end{figure}

The free energy $G_s$ of a secondary structure $s$ is the sum of the energies
of its constituent Loops $L$ \cite{Mathews:99}. Thus the total free energy
$G_s$ is given by $G_s=\sum_{L \in s} G_L$. This notion of energy allows one to
compute, by means of dynamic programming \cite{Hofacker:94a, Zuker:81}
the minimum free energy
configurations as well as the partition function \cite{McCaskill:90} as a
weighted sum $Q=\sum_{s \in \mathcal{S}} e^{G_s/RT}$, where R is the universal
gas constant, $T$ is the temperature and $e^{G_s/RT}$ is the weight of
sampling a secondary structure $s$ with free energy $G_s$.

In the following we consider a notion of energy that does not take into
account the specifics of nucleotides and only depends on the combinatorial
representation the secondary structure.
Our energy model is based on seven parameters
$\mathcal{P}=\{\alpha_1, \alpha_2, \alpha_3, \beta_1, \beta_2,
\gamma_1, \gamma_2\}$ defined as follows:
\begin{itemize}
\item {\sf hairpin-loop} $H$:
\begin{equation}
G_{H} \sim \alpha_1 + \alpha_2\,\ell_{H}+ \alpha_3 \quad and \quad
Q_{H} = v^{G_{H}},
\end{equation}
where $G_{H}$ is the free energy of $H$, $\alpha_1$ is the penalty for
forming $H$, $\alpha_2$ is the penalty for an unpaired base and
$\alpha_3$ is a score associated to a tetra-loop. Furthermore
$\ell_{H}$ denotes the number of unpaired bases and $Q_{H}$ is the
weight of $H$.
\item
{\sf Interior-loop} $I$:
\begin{equation}
G_{I}\sim {\beta_1}+\ell_{I}{\beta_2} \quad and \quad Q_{I} = v^{G_{I}}
\end{equation}
where $G_{I}$ is free energy of $I$, $\beta_1$ is a favorable bonus of
a ${\sf helix}$ and $\beta_2$ is the penalty of an unpaired base of $I$.
Furthermore $\ell_{I}$ denotes the total number of unpaired bases of $I$
and $Q_{I}$ is the weight of $I$.
\item
{\sf Multi-loop} $M$:
\begin{equation}
G_{M} \sim \gamma_1+B \gamma_2 + U \times 0 \quad and \quad
Q_{M} =  v^{G_{M}}
\end{equation}
$G_{M}$ is the free energy of $M$, $\gamma_1$ is the penalty for the
formation of a multiloop, $B$ is the number of base pairs defining
the multiloop (including the closing pair $i \cdot j$), and $U$ is
the number of unpaired bases in the multiloop; $\gamma_2$ is the penalty
for the base pair defining the multiloop and $Q_{M}$ is the weight of one
$M$.
\end{itemize}

This energy model induces the
partition function
\begin{equation}
Q=\sum_{s \in \mathscr{S}(n)} v^{G_s},
\end{equation}
where $v^{G_s}$ the weight for sampling a secondary structure with free
energy $G_s$.

\section{Energy filtration}

\begin{theorem}\label{T:Crecursion}
The trivariate generating function ${\bf C}(z, v, p)$ counting RNA secondary
structures having length $n$, energy $i$ and $j$ arcs, satisfies the following
recursion:
\begin{equation}\label{E:constructC}
\begin{split}
{\bf C}(z, v, p)&=
p\,v^{\alpha_1}z^2 \left(\frac{(zv^{\alpha_2})^3}{1-zv^{\alpha_2}}
+z^4v^{\alpha_3}-z^4v^{4\alpha_2}\right)\\
&+p\,\frac{z^2v^{\beta_1}}{(1-zv^{\beta_2})^2} {\bf C}(z, v, p)
+ p\, \frac{v^{\gamma_1} (z^2v^{\gamma_2}) ({\bf C}(z, v, p) v^{\gamma_2})^2}
{((1-z)^3-{\bf C}(z, v, p) v^{\gamma_2}(1-z)^2 )}.
\end{split}
\end{equation}
\end{theorem}

The proof is a straightforward exercise in symbolic methods and will
consequently be omitted.

We shall pass from the trivariate generating function to univariate
ones by specifying the indeterminants $v$ and $p$ as follows:
since there are a total of four nucleotides ${\bf A}$, ${\bf G}$, ${\bf U}$
and ${\bf C}$ and only six of base pairs, namely ${\bf GC}$,
${\bf CG}$, ${\bf AU}$, ${\bf UA}$, ${\bf UG}$ and ${\bf GU}$,
we set $p=\frac{6}{16}$. This choice reflects the probability of randomly
forming a (valid) base pair. We proceed similarly for the indeterminant $v$
setting $v=e^{\frac{1}{RT}}\approx 1.843868184$.

These two choices induce the
univariate generating functions
${\bf C}^{*}(z)=\sum_{n \geq 0} {\bf c}^{*}(n)z^n$,
${\bf S}^{*}(z)=\sum_{n \geq 0} {\bf s}^{*}(n)z^n$ and the sets of
$\mathcal{S}^{*}$ and $\mathcal{C}^{*}$, which are weighted set
of $\mathcal{S}$ and $\mathcal{C}$, separately. Furthermore,
${\bf c}^{*}(n)$ or ${\bf s}^{*}(n)$ denote the summation of energy weight
of all the $\mathcal{C}^{*}$-structures or $\mathcal{S}^{*}$-structures with length $n$. \\
We have
${\bf c}^{*}(n)=0$, ${\bf s}^{*}(n)=0$, for $n=0, \ldots, 4$; and
${\bf c}^{*}(n)>0$, ${\bf s}^{*}(n)>0$, for $n \geq 5$, see Claim $0$ of
Lemma~\ref{L:Cunique}.

Symbolic methods immediately imply
\begin{theorem}
The bivariate generating function ${\bf S}^*(z)$ is
given by:
\begin{equation}
{\bf S}^*(z)=\frac{1}{1-(z+{\bf C}^*(z))}
\end{equation}
and furthermore,
${\bf S}^*(z,t)$ is
\begin{equation}\label{E:labelC}
{\bf S}^*(z,t)=\frac{1}{1-(z+t\,{\bf C}^*(z))}.
\end{equation}
\end{theorem}

The key for the following analysis is to study the dominant singularities
of ${\bf C}^{*}(z)$ and ${\bf S}^{*}(z)$. It is their relative location and
behavior that is responsible for the observed limit distributions as well
as the effect of sparsification.

We begin our analysis by making the following observations.
Theorem~\ref{T:Crecursion} implies
\begin{equation}\label{E:Ceq}
w^{*}_2(z){\bf C}^{*}(z)^2+ w^{*}_1(z){\bf C}^{*}(z)+w^{*}_0(z)=0,
\end{equation}
where
\begin{equation*}
\begin{split}
w^{*}_2(z)&=(16v^{\gamma_2}(1-z)^2+
6z^2v^{\gamma_1+3\gamma_2})(1-zv^{\alpha_2})(1-zv^{\beta_2})^2\\
&-6z^2v^{\beta_1+\gamma_2}(1-z)^2(1-zv^{\alpha_2})\\
w^{*}_1(z)&=6z^2v^{\beta_1}(1-z)^3(1-zv^{\alpha_2})
-6v^{\alpha_1+\gamma_2+3\alpha_2}z^5(1-zv^{\beta_2})^2(1-z)^2
+6z^6v^{\alpha_1+\gamma_2}\\
&(v^{4\alpha_2}-v^{\alpha_3})(1-z)^2\,(1-zv^{\alpha_2})(1-zv^{\beta_2})^2
-16(1-z)^3(1-zv^{\alpha_2})(1-zv^{\beta_2})^2\\
w^{*}_0(z)&=6v^{\alpha_1}{z^5}(1-z)^3(1-zv^{\beta_2})^2
\left(v^{3\alpha_2}-z(v^{4\alpha_2}-v^{\alpha_3})(1-zv^{\alpha_2})\right).
\end{split}
\end{equation*}
We observe that only one solution of eq.~(\ref{E:Ceq}) has the property
${\bf c}^{*}(n)>0$, namely
\begin{equation}\label{E:Csolution}
{\bf C}^{*}(z)=\frac{N^{*}(z)}{D^{*}(z)},
\end{equation}
where $N^{*}(z)=-w_1^{*}(z)
+ \sqrt{w^{*}_1(z)^2-4w_2^{*}(z)w_0^{*}(z)}$ and
$D^{*}(z)=2w_2^{*}(z)$.

Let $\rho_{c^{*}}$ denote the radius of convergence of
${\bf C}^{*}(z)$ and $\rho_{s^{*}}$ denote the radius of convergence of
${\bf S}^{*}(z)$. Then, clearly, $0<\rho_{c^{*}}<1$ and in view of
$\mathcal{C}^{*}\subset \mathcal{S}^{*}$, we have $0<\rho_{s^{*}} \leq
\rho_{c^{*}}<1$.

$\bullet$ Let furthermore $\rho_{r}$ denote the minimum positive real
root of odd order of the discriminant
$$
w^{*}_1(z)^2-4w^{*}_2(z)w^{*}_0(z)=0.
$$
$\bullet$ Let $\rho_{d}$ and $\rho_{p}$ denote the minimum positive real
      roots of the equations
\begin{equation*}
w^{*}_2(z) = 0 \quad \text{and} \quad 1-(z+{\bf C}^{*}(z))=0.
\end{equation*}

\begin{theorem}
{\bf (Pringsheim$'$s theorem)}\label{T:l}\cite{Titchmarsh}:
If $f(z)$ is representable at the origin by a series expansion that has
non-negative coefficients and radius of
convergence $r$, then the point $z = r$ is a dominant singularity of $f(z)$.
\end{theorem}

\begin{lemma}\label{L:singular}
 For the positive dominant real singularities of ${\bf C}^*(z)$ and
${\bf S}^*(z)$,
$\rho_{c^{*}}$ and $\rho_{s^{*}}$, we distinguish the following three cases:\\
{\bf (I):} $\rho_{c^{*}}=\rho_{r}$ and $\rho_{s^{*}}=\rho_{c^{*}}$;\\
{\bf (II):} $\rho_{c^{*}}=\rho_{r}$ and $\rho_{s^{*}}=\rho_{p}$,
        where $\rho_{s^{*}}<\rho_{c^{*}}$;\\
{\bf (III):} $\rho_{c^{*}}=\rho_{r}$ and $\rho_{s^{*}}=\rho_{p}$,
        where $\rho_{s^{*}}=\rho_{c^{*}}$.
\end{lemma}
\begin{proof}
According to Theorem~\ref{T:l}, we have $\rho_{c^*}=\min\{\rho_{r}, \rho_{d}\}$.

{\it Claim $0$:} We have $\rho_{d}\neq \rho_{r}$. \\
To prove Claim $0$.
We begin by observing that $v^{4\alpha_2}-v^{\alpha_3}<0$,
since $\alpha_2<0$, $\alpha_3>0$ and $v>1$,
then for arbitrary real $z$, $0<z<1$,
\begin{eqnarray}
w^{*}_0(z)&=&6v^{\alpha_1}z^5 \underbrace{(1-z)^3(1-z v^{\beta_2})^2}_{>0}
(v^{3\alpha_2}-z\underbrace{(v^{4\alpha_2}-v^{\alpha_3})}_{<0}
\underbrace{(1-z v^{\alpha_2})}_{>0})>0. \nonumber
\end{eqnarray}
Suppose $\rho_{d} = \rho_{r}$. Then, in view of
eq.~(\ref{E:Csolution})
$$
w^{*}_2(\rho_{d})=0\quad\text{\rm and}\quad
w^{*}_1(\rho_{d})^2-4w^{*}_2(\rho_{d})w^{*}_0(\rho_{d})=0
$$
and consequently $w^{*}_1(\rho_{d})=0$. But in this case
eq.~(\ref{E:Ceq}) implies $w^{*}_0(\rho_{d})=0$,
which conflict with $w^{*}_0(\rho_{d})>0$ and Claim $0$ is proved.

First, we discuss the case $\rho_{d}<\rho_{r}$.\\
{\it Claim $1$:} $\rho_{d}$ is a removable singularity of ${\bf C}^{*}(z)$. \\
To prove {\it Claim $1$}. Suppose first $N^{*}(\rho_{d}) \neq 0$,
since $D^{*}(\rho_{d}) = 0$, thus
$\lim_{z\rightarrow \rho_{d}} {\bf C}^{*}(z)=\infty$, so $\rho_{d}$ is
accordingly a pole of ${\bf C}^{*}(z)$.\\
In view of
\begin{equation*}\label{E:Ssolution}
{\bf S}^{*}(z)^{-1}=1-\left(z+\frac{N^{*}(\rho_{d})}{D^{*}(\rho_{d})}\right),
\end{equation*}
we have $\lim_{z\rightarrow \rho_{d}} {\bf S}^{*}(z)=0$, this is impossible,
since ${\bf S}^{*}(z)=\sum_{n \geq 0} {\bf s}^{*}(n)z^n$, the only value
of $\rho_{d}$ to make $\lim_{z\rightarrow \rho_{d}} {\bf S}^{*}(z)=0$
is $\rho_{d}=0$.\\
Thus $N^{*}(\rho_{d})= 0$ holds,
we compute using $D^{*}(\rho_{d})=2w_2^{*}(\rho_{d})=0$,
\begin{eqnarray*}
N^{*}(\rho_{d})&=&-w_1^{*}(\rho_{d})
+ \sqrt{w^{*}_1(\rho_{d})^2-4w_2^{*}(\rho_{d})w_0^{*}(\rho_{d})}\nonumber \\
&=& -w_1^{*}(\rho_{d})+ {\sf csgn}(w_1^{*}(\rho_{d}))\cdot (w_1^{*}(\rho_{d})) = 0.
\end{eqnarray*}
Consequently, ${\sf csgn}(w_1^{*}(\rho_{d}))$ is $+1$ and we have
\begin{eqnarray*}
&&\lim_{z \rightarrow \rho_{d}} {\bf C}^{*}(z)\\
&= & \lim_{z \rightarrow \rho_{d}} \frac{(-w^{*}_1(z)+\sqrt{w^{*}_1(z)^2-4w^{*}_2(z)
w^{*}_0(z)}) (-w^{*}_1(z)-\sqrt{w^{*}_1(z)^2-4w^{*}_2(z)
w^{*}_0(z)})}{2w^{*}_2(z)(-w^{*}_1(z)-\sqrt{w^{*}_1(z)^2-4w^{*}_2(z)
w^{*}_0(z)}) } \\
& = & \lim_{z \rightarrow \rho_{d}}
 \frac{2w^{*}_0(z)}{(-w^{*}_1(z)-\sqrt{w^{*}_1(z)^2-4w^{*}_2(z)w^{*}_0(z)}) } \\
& = &
\frac{2w^{*}_0(\rho_{d})}{(-w^{*}_1(\rho_{d})- w_1^{*}(\rho_{d}))}=-\frac{w^{*}_0(\rho_{d})}{w_1^{*}(\rho_{d})}.
\end{eqnarray*}
Since $w_2^{*}(\rho_{d})=0$ and $\rho_{d} \neq \rho_{r}$, hence
$ w_1^{*}(\rho_{d})^2-4w_0^{*}(\rho_{d})w_2^{*}(\rho_{d})\neq 0$,
whence $w_1^{*}(\rho_{d}) \neq 0$. Using our previous observation
$w_0^{*}(\rho_{d}) > 0$, we conclude
$$
\lim_{z \rightarrow \rho_{d}} {\bf C}^{*}(z)=
-\frac{w^{*}_0(\rho_{d})}{w_1^{*}(\rho_{d})}
\neq 0, \infty,
$$
which implies that $\rho_{d}$ is a removable singularity, whence {Claim $1$}.

{\it Claim $2$:} $z=\rho_{d}\,e^{i\theta}$, where $0 < \theta < 2\pi$, is not
a dominant singularity of ${\bf C}^{*}(z)$.\\
To prove {\it Claim $2$} we assume a contrario that
$z=\rho_{d}\,e^{i\theta}$, $0 < \theta < 2\pi$ is a dominant singularity of
${\bf C}^{*}(z)$. Then the convergence radius of ${\bf C}^{*}(z)$ is $\rho_{d}$
and Theorem~\ref{T:l} implies that $\rho_{d}$ is also a dominant
singularity of ${\bf C}^{*}(z)$, which contradicts {\it Claim $1$}, where
we showed that $\rho_d$ is a removable singularity.

Therefore, in case of $\rho_{d}<\rho_{r}$,  $\rho_{c^{*}}\neq \rho_{d}$.
Consequently, Claim $1$ and Claim $2$ imply $\rho_{c^*}=\rho_{r}$.

We thus have the following two scenarios for $\rho_{s^*}$:
\begin{equation}\label{E:sqrt}
\rho_{s^*}=
\begin{cases}
\rho_{r} & \text{for }\ \rho_{r}<\rho_{p}  \\
\rho_{p} & \text{for }\ \rho_{p}\le \rho_{r} .
\end{cases}
\end{equation}
and obtain the three cases:\\
{\bf (I):}  $\rho_{c^*}=\rho_{r}$ and $\rho_{s^*}=\rho_{r}$;\\
{\bf (II):}
$\rho_{c^*}=\rho_{r}$ and $\rho_{s^*}=\rho_{p}$ with $\rho_{p}<\rho_{r}$;\\
{\bf (III):}
$\rho_{c^*}=\rho_{r}$ and $\rho_{s^*}=\rho_{p}$ with $\rho_{p}=\rho_{r}$\\
and Lemma~\ref{L:singular} follows.
\end{proof}

\begin{lemma}\label{L:Cunique}
The dominant singularity of ${\bf C}^{*}(z)$, $\rho_{c^*}$, is unique.
\end{lemma}
\begin{proof}
{\it Claim $0$}: ${\bf C}^{*}(z)$ is aperiodic.\\
It is clear that there always exists some irreducible secondary structure
of length $n\ge 5$.
The coefficients of ${\bf C}^*(z)$ are weighted sums of these structures and
the weights are, by construction, always strictly positive. Therefore,
${\bf C}^{*}(z)=\sum_{n \geq 0} {\bf c}^{*}(n) z^n$ has for $n\ge 5$ always
strictly positive coefficients, i.e.~${\bf c^*}(n)>0$.
Hence, there exist three indices $i<j<k$ such that
${\bf c}^{*}(i){\bf c}^{*}(j){\bf c}^{*}(k) \neq 0$ and
$\gcd(j-i, k-i)=1$, therefore ${\bf C}^{*}(z)$
is aperiodic and {\it Claim $0$} is proved.

In view of eq.~(\ref{E:constructC}), we compute
\begin{eqnarray}\label{E:pC}
& &{\bf C}(z,v,p) \nonumber\\
&=& pv^{\alpha_1}z^2\,(zv^{\alpha_2})^3\,\sum_{i \geq 0 } (zv^{\alpha_2})^i
+(v^{\alpha_3}-v^{4\alpha_2})\,z^4+ pz^2v^{\beta_1}\,{\bf C}(z,v,p)
(\sum_{j \geq 0} (zv^{\beta_2})^j)^2 \nonumber\\
& & + pv^{\gamma_1}(z^2v^{\gamma_2})({\bf C}(z,v,p)v^{\gamma_2})^2\,
(\sum_{k \geq 0} z^k) \sum_{l \geq 0} \left({\bf C}(z,v,p) v^{\gamma_2}
\sum_{t \geq 0} z^t \right)^{l}.
\end{eqnarray}
Setting $p=\frac{6}{16}$, $v=e^{\frac{1}{RT}}$ and
$w={\bf C}^{*}(z)=\sum_{n \geq 0} {\bf c^{*}}(n) z^n$ we obtain
a power series equation
\begin{equation}\label{E:Eq}
w=G(z, w),
\end{equation}
where $G(z, w)=\sum_{m,n>0}g_{m,n} z^m w^n$ and $g_{m,n} \geq 0$. Indeed,
since $\alpha_3 >0$ and $\alpha_2 <0$, we have $v^{4\alpha_2}<v^{\alpha_3}$
and $v^{\alpha_3}-v^{4\alpha_2}>0$. The other coefficients in
eq.~(\ref{L:Cunique}) are all positive, i.e.
$v^{\mathcal{P}}>0$, $\mathcal{P}=\{\alpha_1, \alpha_2, \alpha_3,
\beta_1, \beta_2,\gamma_1, \gamma_2\}$, implying $g_{m,n} \geq 0$,
in particular, $g_{0,1} =0$.

Furthermore, $G(z, w)$ is
bivariate power series which is absolutely convergent in a domain
$\mathscr{D}$, such that $|z|<R_1$, $|w|<R_2$.
According to {Theorem 9.4.4} of \cite{E.Hille}, there exist a unique
function analytic in a neighborhood $|z|<\rho$ of $z=0$, where $\rho
\leq R_1$, such that
\begin{equation*}
{\bf C}^{*}(0)=0 \quad \text{and} \quad
G(z, {\bf C}^{*}(z))-{\bf C}^{*}(z)=0, \quad \text{for} \quad |z|<\rho,
\quad  {\bf C}^{*}(z)<R_2.
\end{equation*}
Furthermore {Theorem 9.4.6} of \cite{E.Hille} shows that the radius of
convergence, $\rho=\rho_{c^*}$, of the solution of eq.~(\ref{E:Eq}),
${\bf C}^{*}(z)$, and the value $w_{\rho}=\lim_{z\rightarrow \rho}{\bf C}^{*}(z)$
satisfy the equations
\begin{equation*}
w_{\rho}=G(\rho, w_{\rho}) \quad \text{and} \quad 1=G_{w_{\rho}}(\rho, w_{\rho}).
\end{equation*}
According to {Lemma}~\ref{L:singular}, we have $\rho_{c^{*}}=\rho_{r}$ and thus
\begin{eqnarray}
w_{\rho}=\lim_{z\rightarrow \rho}{\bf C}^{*}(z)={\bf C}^{*}(\rho_{r})=
\frac{-w^{*}_1(\rho_{r})}{w^{*}_2(\rho_r)}\neq 0, +\infty.
\end{eqnarray}

We next show that ${\bf C}^{*}(z)$ has no other dominant singularities than
$\rho_{r}$.

To this end we note that ${\bf C}^{*}(z)$ converges at point $z=\rho_{r}$.
Since ${\bf c^{*}}(n) \geq 0$, applying the triangular inequality, we have
${\bf C}^{*}(\rho_r \, e^{i \theta}) \leq {\bf C}^{*}(\rho_{r})$ Therefore
${\bf C}^{*}(z)$ converges on the whole circle $|z|=\rho_{r}$.
Since ${\bf C}^{*}(z)$ is aperiodic and ${\bf C}^{*}(z)=\sum_{n \geq 0}
{\bf c}^{*}(n)z^n$ is convergent power series for any $z$ with
$|z|=\rho_{r}$, the {Daffodil Lemma} of \cite{Flajolet:07} implies
$$
|{\bf C}^{*}(\rho_{r}\,e^{i \theta})| < {\bf C}^{*}(\rho_r).
$$
Taking the derivative in eq.~(\ref{E:Eq}), we derive
$$
\frac{d}{dz}{\bf C}^{*}(z)=
\frac{d}{d {\bf C}^{*}}G(z, {\bf C}^{*}(z))\frac{d}{dz}{\bf C}^{*}(z)
+ \frac{d}{d z}G(z, {\bf C}^{*}(z)).
$$
Thus we have
\begin{equation}
w_z=\frac{G_z(z, w)}{1-G_w(z, w)},
\end{equation}
which implies that ${\bf C}^{*}(z)$ is indeed analytic as long as
$G_{w}(z, w) \neq 1$.

Since $G(z,w)$ has non-negative coefficients, it is monotonously
increasing and so is $G_w(z,w)$. Therefore, for points $\rho_{r}\,
e^{i \theta}$, where $0<\theta<2\pi$ we have
\begin{eqnarray*}
& &|G_w(\rho_{r}\,e^{i \theta}, {\bf C}^{*}(\rho_{r}\,e^{i \theta}))|
\leq |G_w(|\rho_{r}\,e^{i \theta}|, |{\bf C}^{*}(\rho_{r}\,e^{i \theta})|)|
<|G_w(\rho_{r}, {\bf C}^{*}(\rho_{r}))| =1,
\end{eqnarray*}
which implies that ${\bf C}^{*}(z)$ is analytic on the whole circle
$|z|=\rho_{r}$ except of $\rho_{r}$. This proves that
$\rho_{c^*}=\rho_r$ is unique.
\end{proof}

\begin{lemma}\label{L:singular-S}
The dominant singularity of ${\bf S}^{*}(z)$ is unique and there are
the following three cases:\\
{\bf (I):} If $\rho_{s^*}=\rho_{c^*}=\rho_{r}$,
           then $\rho_{r}$ is the unique dominant singularity of
           ${\bf S}^{*}(z)$;\\
{\bf (II):} If $\rho_{s^*}=\rho_{p}<\rho_{r}$, then
           $\rho_{p}$ is the unique dominant singularity of
           ${\bf S}^{*}(z)$;\\
{\bf (III):} If $\rho_{s^*}=\rho_{p}=\rho_{r}$, then
           $\rho_{p}$ is the unique dominant singularity of
           ${\bf S}^{*}(z)$;\\
Furthermore, in case of {\bf (II)} and {\bf (III)}, the degree of
the root, $\rho_{p}$ of the equation $1-(z+{\bf C}^{*}(z))=0$ is
exactly $1$.
\end{lemma}
\begin{proof}
According to {Lemma}~\ref{L:singular}, we need to distinguish the
cases {\bf (I)}, {\bf (II)} and {\bf (III)}.

{\bf (I)}, here $\rho_{r}$ is the unique dominant singularity of
${\bf C}^{*}(z)$.
Then in view of eq.~(\ref{E:Ssolution}), the dominant singularity of
${\bf S}^{*}(z)$ comes from the circle $|z|=\rho_{r}$, whence $\rho_{r}$ is
also the unique dominant singularity of ${\bf S}^{*}(z)$.

{\bf (II)} in this case we have $\rho_{s^*}=\rho_{p}<\rho_{c^*}$ and
the dominant singularities of ${\bf S}^{*}(z)$ all come from the
circle $|z|=\rho_{p}$.
We have
\begin{equation*}\label{E:Ssolu}
{\bf S}^{*}(z)^{-1}=1-\left(z+{\bf C}^{*}(z)\right)
\end{equation*}
and substituting $z=\rho_{p}$, we derive
\begin{equation*}
\left(\rho_{p}+{\bf C}^{*}(\rho_{p})\right)=1.
\end{equation*}
Then ${\bf D}^{*}(z)=z+{\bf C}^{*}(z)=\sum_{n \geq 0} {\bf d^{*}}(n) z^n$
is an aperiodic power series with non-negative coefficients.
According to {Lemma}~\ref{L:singular}, we have
\begin{equation}
{\bf D}^{*}(\rho_{p})=\rho_{p}+{\bf C}^{*}(\rho_{p})=1.
\end{equation}
Since ${\bf d^{*}}(n) \geq 0$, applying the triangular inequality implies
${\bf D}^{*}(\rho_p \, e^{i \theta}) \leq {\bf D}^{*}(\rho_{p})$ and
${\bf D}^{*}(z)$ converges on the whole circle $|z|=\rho_{p}$.
Since ${\bf D}^{*}(z)$ is aperiodic the {Daffodil Lemma} of
\cite{Flajolet:07}, guarantees
$$
|{\bf D}^{*}(\rho_{p}\,e^{i \theta})| < {\bf D}^{*}(\rho_p)<1.
$$
Therefore, $|{\bf D}^{*}(\rho_{p}\,e^{i \theta})| \neq 1$, for $0<\theta<2\pi$,
whence the points on the circle $|z|=\rho_{p}$ other than the point
$\rho_{p}$, are not singularities of ${\bf S}^{*}(z)$.

{\bf (III)} here the dominant singularities of ${\bf S}^{*}(z)$ all come
from the circle $|z|=\rho_{p}=\rho_{r}$. According {\bf (I)} and {\bf (II)},
$\rho_{p}=\rho_{r}$ is an unique dominant singularity of ${\bf S}^{*}(z)$.

In case of {\bf (II)} and {\bf (III)} and in view of eq.~(\ref{E:Ssolu}),
the denominator of ${\bf S}^{*}(z)$ is
\begin{equation}\label{E:denominate}
1-(z+{\bf C}^{*}(z))
\end{equation}
and $1-(\rho_{p}+{\bf C}^{*}(\rho_{p}))=0$. Taking the derivative of
eq.~(\ref{E:denominate}), we obtain $(1+ \frac{d}{d z}{\bf C}^{*}(z))=0$.
Since ${\bf C}^{*}(z)=\sum_{n \geq 0} {\bf c^{*}}(n) z^n$, where
${\bf c^{*}}(n) > 0$, for $n\geq 5$, the derivative
$\frac{d}{d z}{\bf C}^{*}(z)=
\sum_{n \geq 1} n\,{\bf c^{*}}(n) z^{n-1}$, has also nonnegative coefficients.
As a result we have $\frac{d}{d z}{\bf C}^{*}(\rho_{p})>0$ and consequently
$(1+ \frac{d}{d z}{\bf C}^{*}(\rho_{p}))>0$. We conclude that $\rho_{p}$ is
not a multiple root of eq.~(\ref{E:denominate}), which completes the proof
of {Lemma}~\ref{L:singular-S}.
\end{proof}

\section{The main result}
We consider the general composition scheme
\begin{equation}
\mathcal{F}=\mathcal{G}\circ(u \mathcal{H})
\Longrightarrow {\bf F}(z,u)=g(uh(z)).
\end{equation}
Assume that $g$ and $h$ have non-negative coefficients and
that $h(0)=0$, so that the composition $g(h(z))$ is well-defined.
We let $\rho_g$ and $\rho_h$ denote the radii of convergence of
$g$ and $h$, and define
\begin{equation}
\tau_g=\lim_{x\rightarrow \rho_g^{-}}\, g(x) \quad and \quad
\tau_h=\lim_{x\rightarrow \rho_h^{-}}\, h(x).
\end{equation}
\begin{definition}\cite{Flajolet:07}
The composition scheme ${\bf F}(z,t)=g(th(z))$ is said to be
subcritical if $\tau_{h}<\rho_{g}$, critical if $\tau_{h}=\rho_{g}$,
and supercritical if $\tau_h>\rho_g$.
\end{definition}
We observe that
\begin{equation}\label{E:eumel}
{\bf S}^{*}(z,t)=\frac{1}{1-(z+t\,{\bf C}^{*}(z))}
=f(z)\,g(t\,h(z)),
\end{equation}
where $f(z)=\frac{1}{1-z}$, $g(w)=\frac{1}{1-w}$ and
$h(z)=\frac{{\bf C}^{*}(z)}{1-z}$. Furthermore, we set
$$
\mathbb{P}(X_n=t)={\bf s}^{*}(n,t)/{\bf s}^{*}(n).
$$
Since $0<\rho_{s^{*}}\leq \rho_{c^{*}}<1$, we restrict ourselves to
the cases $0<\rho_{d}, \rho_{r}, \rho_{p}<1$.

\begin{theorem}\label{T:quotient}
The distribution of irreducible substructures within energy-filtered
RNA secondary structures has the following distinct regimes:\\
{\bf (a) the subcritical regime:} Both dominant singularities
$z=\rho_{c^{*}}$ of ${\bf C}^{*}(z)$ and $z=\rho_{s^{*}}$ of ${\bf S}^{*}(z)$
are all exclusively a branch point (square root) singularity. Then
$\mathbb{P}(X_n=t)$ satisfies a discrete limit
law and
$$
\lim_{n \rightarrow \infty} \frac{{\bf c}^{*}(n)}{{\bf s}^{*}(n)} =\chi >0.
$$
{\bf (b) the supercritical regime:} The dominant singularity
$z=\rho_{c^{*}}$ of ${\bf C}^{*}(z)$ is exclusively branch point singularity;
the dominant singularity $z=\rho_{s^{*}}$ of ${\bf S}^{*}(z)$ is
exclusively a pole; .
Then the probability distribution of $\mathbb{P}(X_n=t)$, after
standardization, satisfies a limiting Gaussian distribution and
$$
\lim_{n \rightarrow \infty}\frac{{\bf c}^{*}(n)}{{\bf s}^{*}(n)}=0.
$$
{\bf (c) the critical regime:} The dominant singularity
$z=\rho_{c^{*}}$ of ${\bf C}^{*}(z)$ is exclusively branch point singularity,
the dominant singularity $z=\rho_{s^{*}}$ of ${\bf S}^{*}(z)$ is
simultaneously a branch point singularity and a pole. Then
$\mathbb{P}(X_n=t)$ satisfies a local limit law whose
density is a Rayleigh distribution and
$$
\lim_{n \rightarrow \infty} \frac{{\bf c}^{*}(n)}{{\bf s}^{*}(n)}\sim
\frac{\chi^{'}}{n},
$$
where $\chi'>0$.
\end{theorem}

\begin{proof}
To prove the theorem we shall show how the three scenarios identified
in {Lemma}~\ref{L:singular-S} give rise to the three regimes.

{\bf (a):} Let us begin with case {\bf (I)} of {Lemma}~\ref{L:singular-S},
where we have $\rho_{c^{*}}=\rho_{s^{*}}
=\rho_{r}$ and $\rho_{r}<\rho_{p}$ and both dominant singularities
$z=\rho_{c^{*}}$ of ${\bf C}^{*}(z)$ and $z=\rho_{s^{*}}$ of ${\bf S}^{*}(z)$
are all exclusively a branch point (square root) singularity.\\
We set
\begin{eqnarray*}
t_0(z) & = & (w^{*}_1(z)^2-4w_2^{*}(z)w_0^{*}(z))/(z-\rho_{r}),\\
t_1(z) & = & (-w_1^{*}(z))/(2w_2^{*}(z)),\\
t_2(z) & = & (\sqrt{t_0(z)})/(2w_2^{*}(z))
\end{eqnarray*}
and express ${\bf C}^{*}(z)$ as
\begin{equation}
\begin{split}
{\bf C}^{*}(z)=t_1(z)+t_2(z)(z-\rho_{r})^{\frac{1}{2}}.
\end{split}
\end{equation}
The singular expansion of  ${\bf C}^{*}(z)$ is obtained from
the regular expansion of $t_1(z)$ and the singular expansion of
$(z-\rho_{r})^{\frac{1}{2}}t_2(z)$. Consequently,
\begin{equation}\label{E:singu-C-st}
{\bf C}^{*}(z)=t_1(\rho_{r})+t_2(\rho_{r})(z-\rho_{r})^{\frac{1}{2}}
+O(z-\rho_{r}).
\end{equation}
In view of $O(z-\rho_{r})=o((z-\rho_{r})^{\frac{1}{2}})$,
{Theorem VI.3} \cite{Flajolet:07} implies
\begin{equation}
[z^n]{\bf C}^{*}(z) \sim t_2(\rho_{r})[z^n](z-\rho_{r})^{\frac{1}{2}}.
\end{equation}
Using {Theorem VI.1} of \cite{Flajolet:07}, we obtain
\begin{equation}
[z^n]{\bf C}^{*}(z) \sim k_1 \cdot n^{-\frac{3}{2}}
\cdot(\rho_{r})^{-n}\cdot(1+O(\frac{1}{n})), \quad \text{\rm for
some constant $k_1$}.
\end{equation}
Analogously we compute the singular expansion of ${\bf S}^{*}(z)$ as:
\begin{equation}
{\bf S}^{*}(z)=d_1+d_2(z-\rho_{r})^{\frac{1}{2}}+O(z-\rho_{r}),
\quad \text{\rm for some constants $d_1,d_2$}.
\end{equation}
Employing {Theorem VI.1} of \cite{Flajolet:07},
\begin{equation}
[z^n]{\bf S}^{*}(z) \sim k_2 \cdot n^{-\frac{3}{2}}
\cdot(\rho_{r})^{-n}(1+O(\frac{1}{n})), \quad \text{\rm for
some constant $k_2$}.
\end{equation}
Therefore,
\begin{equation}
\lim_{n\rightarrow \infty}\frac{{\bf c}^{*}(n)}{{\bf s}^{*}(n)}
\sim
\frac{ k_1\cdot n^{-\frac{3}{2}}
\cdot (\rho_{r})^{-n} }
{k_2\cdot n^{-\frac{3}{2}}
\cdot (\rho_{r})^{-n}}=\chi>0,
\end{equation}
We now have $\rho_{r}<\rho_{p}$, furthermore, we have (eq.~(\ref{E:eumel}))
with $\rho_{g}=1$ and
$$
\tau_{h}=\lim_{z \rightarrow \rho_{h}^{-}} \frac{{\bf C}^{*}(z)}{1-z}.
$$
Since $0<\rho_{p}<1$ is a root of $1-(z+{\bf C}^{*}(z))=0$, we observe
$\rho_{p}+{\bf C}^{*}(\rho_{p})=1$.
${\bf C}^{*}(z)$ is a power series with positive coefficients and thus
as a function over $[0,1[$, continuous and monotone.
As a result $\rho_{r}+{\bf C}^{*}(\rho_{r})<1$ and
$\tau_{h}=h(\rho_{r})=\frac{{\bf C}^{*}(\rho_{r})}{1-\rho_{r}}<1=\rho_{g}$,
i.e.~${\bf S}^{*}(z,t)$ is governed by the subcritical paradigm.

According to {\rm Proposition IX.1} \cite{Flajolet:07}, the
quotient of the coefficients ${\bf s}^{*}(n,k)$ and ${\bf s}^{*}(n)$ satisfies
\begin{equation}
\lim_{n\rightarrow \infty} \frac{{\bf s}^{*}(n,k)}{{\bf s}^{*}(n)}=q_k,
\ {\rm where} \ q_k=\frac{k g_k \tau_h^{k-1}}{g^{'}(\tau_h)}.
\end{equation}
Therefore the probability generating function of the limit distribution
$(q_k)$ is given by
\begin{equation}
q(t)=\frac{tg^{'}(\tau_h t)}{g^{'}(\tau_h)}
\end{equation}
and $\mathbb{P}(X_n=t)=\frac{{\bf s}^{*}(n,k)}{{\bf s}^{*}(n)}$
satisfies a discrete limit law as asserted.


Ad {\bf (b):} We have the case {\bf (II)} of {Lemma}~\ref{L:singular-S},
we have $\rho_{c^{*}}=\rho_{r}$, $\rho_{s^{*}}=\rho_{p}$ and
$\rho_{p}<\rho_{r}$,
the dominant singularity $z=\rho_{c^{*}}$ of ${\bf C}^{*}(z)$ is
exclusively branch point singularity; the dominant singularity
$z=\rho_{s^{*}}$ of ${\bf S}^{*}(z)$ is exclusively a pole with the relation.

In analogy to our arguments in {\bf (a)} we derive
\begin{eqnarray*}
{\bf s}^{*}(n) & \sim &
k_3 \cdot 1 \cdot(\rho_{s^{*}})^{-n}\cdot(1+O(\frac{1}{n})),
\quad \omega_1 \geq 1 \, \, {\rm and} \, \, \omega_1 \in \mathbb{N},  \\
{\bf c}^{*}(n) & \sim &  k_1 n^{-\frac{3}{2}}
\cdot(\rho_{c^{*}})^{-n} \cdot (1+O(\frac{1}{n})),
\end{eqnarray*}
for constants $k_3,k_1$ and clearly,
\begin{equation}
 \frac{{\bf c}^{*}(n)}{{\bf s}^{*}(n)}
\sim \frac{k_1 \cdot n^{-\frac{3}{2}} (\rho_{c^{*}})^{-n}}
{k_3\cdot 1 \cdot (\rho_{s^{*}})^{-n} }
=\frac{k_1}{k_3} \cdot n^{-\frac{3}{2}}
\left(\frac{\rho_{s^{*}}}{\rho_{c^{*}}}\right)^n \sim 0.
\end{equation}

We now have $\rho_{c^{*}}> \rho_{p}=\rho_{s^{*}}$ and since for
$0<\rho_{p}<1$ as well as $\rho_{p}+{\bf C}^{*}(\rho_{p})=1$,
analogous monotonicity arguments imply
$\rho_{c^{*}}+{\bf C}^{*}(\rho_{c^{*}})>1$.
As a result
$$
\tau_{h}=h(\rho_{c^{*}})=\frac{{\bf C}^{*}(\rho_{c^{*}})}{1-\rho_{c^{*}}}>1=\rho_{g},
$$
i.e.~${\bf S}^{*}(z,t)$ is governed by the supercritical paradigm.
According to {\rm Proposition IX.6} \cite{Flajolet:07},
$\mathbb{P}(X_n=k)={\bf s}^{*}(n,k)/{\bf s}^{*}(n)$ satisfies a Gaussian
limit law with speed of convergence $O(1/\sqrt{n})$.


Ad {\bf (c)}: We consider finally case {\bf (III)} of
{Lemma}~\ref{L:singular-S}, where we have
$\rho_{c^{*}}=\rho_{r}$ and $\rho_{s^{*}}=\rho_{p}$ with $\rho_{p}=\rho_{r}$,
the dominant singularity
$z=\rho_{c^{*}}$ of ${\bf C}^{*}(z)$ is exclusively branch point singularity,
the dominant singularity $z=\rho_{s^{*}}$ of ${\bf S}^{*}(z)$ is
simultaneously a branch point singularity and a pole.\\
Then the singular expansion of ${\bf S}^{*}(z)$ is given by:
\begin{equation}
{\bf S}^{*}(z)=d_3+d_4\, (z-\rho_{s^{*}})^{-\frac{1}{2}}
+O(z-\rho_{s^{*}})
\end{equation}
for constant $d_3$ and $d_4$.
 We compute
\begin{equation}
{\bf s}^{*}(n) \sim
k_4\cdot n^{-\frac{1}{2}} \cdot (\rho_{s^{*}})^{-n}
\cdot (1+O(\frac{1}{n}))
\end{equation}
and for constant $k_4$.
For ${\bf c}^{*}(n)$ we derive the asymptotic expression
\begin{equation}
{\bf c}^{*}(n) \sim k_1 \cdot n^{-\frac{3}{2}}
\cdot(\rho_{c^{*}})^{-n}\cdot(1+O(\frac{1}{n})).
\end{equation}
Consequently,
\begin{equation}
 \frac{{\bf c}^{*}(n)}{{\bf s}^{*}(n)}
\sim \frac{k_1 \cdot n^{-\frac{3}{2}} \cdot (\rho_{c^{*}})^{-n}}
{k_4 \cdot n^{-\frac{1}{2}} \cdot (\rho_{s^{*}})^{-n} }
=\frac{\chi^{'}}{n} \sim 0,
\,\, \chi^{'}=\frac{k_1}{k_4}.
\end{equation}

In view of $\rho_{r}=\rho_{p}$ we obtain
$\frac{{\bf C}^{*}(\rho_r)}{1-\rho_r}=1$ and
$$
\tau_{h}=h(\rho_{r})=\frac{{\bf C}^{*}(\rho_{r})}{1-\rho_{r}}=1=\rho_{g},
$$
whence ${\bf S}^{*}(z,t)$ belongs to the critical paradigm. The critical
composition is ${\bf S}^{*}(z,t)=f(z)\,g(t\,h(z))$, where $h(z)$
and $g(z)$ have singular exponent $\lambda=\frac{1}{2}$ and $\lambda^{'}=1$,
where $\lambda<\lambda^{'}$. We proceed by applying {\rm Proposition IX.24} of
\cite{Flajolet:07}, from which we can conclude that the normalized
r.v.~$X_n/\sqrt{n}$ satisfies a local limit law whose density is given by a
Rayleigh law \cite{ Banderier, Flajolet:07}. To be specific we have
\begin{equation}
\mathbb{P}(X_n=k)=\frac{[z^k]g(z)}{[z^n]g(h(z))}[z^n]h(z)^k,
\end{equation}
and the singular expansion of $g(z)$, $h(z)$ and $g(h(z))$ are respectively
given by
\begin{eqnarray*}
g(z) & = & (1-z)^{-1}, \\
h(z) & = & \tau_{h}-h_{\frac{1}{2}}(1-\frac{z}{\rho_{r}})^{\frac{1}{2}}+
         O(1-\frac{z}{\rho_{r}}),\\
g(h(z)) & = & m_0-m_{-\frac{1}{2}} (1-\frac{z}{\rho_{r}})^{-\frac{1}{2}}+O(1).
\end{eqnarray*}
Thus we obtain
\begin{equation}
[z^k]g(z)\sim (1)^{-k} \quad \text{\rm and} \quad
[z^n] g(h(z)) \sim m_{-\frac{1}{2}}\cdot \frac{n^{-\frac{1}{2}}}{\Gamma(\frac{1}{2})}
\cdot (\rho_{r})^{-n}.
\end{equation}
We next employ eq.~(103) of {\rm Theorem IX.16} of \cite{Flajolet:07} for
$k=xn^{\frac{1}{2}}$, where $x$ is contained in any compact subinterval of
$(0, +\infty)$. Then
\begin{equation}
[z^n]h(z)^k \sim (\rho_{r})^{-n}\cdot \frac{1}{n} \cdot
             {\rm Ray}(h_{\frac{1}{2}}\,x; \frac{1}{2}),
\end{equation}
where ${\rm Ray}(x; \frac{1}{2})=\frac{x}{2}\,\exp(-\frac{x^2}{4})$
is the Rayleigh density function. Hence
\begin{eqnarray*}
\mathbb{P}(X_n=k) & \sim & \frac{\Gamma(\frac{1}{2})}{m_{-\frac{1}{2}}}\,
\frac{1}{\sqrt{n}}\,{\rm Ray}(h_{\frac{1}{2}}\,x; \frac{1}{2})\\
& = & \frac{h_{\frac{1}{2}} \Gamma(\frac{1}{2})}{2m_{-\frac{1}{2}}} \cdot \frac{k}{n}
\cdot \exp(-\frac{(h_{\frac{1}{2}})^2}{4}\,\frac{k^2}{n})
\end{eqnarray*}
and the proof of the theorem is complete.
\end{proof}

\section{Discussion}

In this paper we demonstrated that the particular energy parametrization
of RNA secondary structures affects the class of minimum free energy
structures generated by DP-mfe folding algorithms in a subtle way.
Minimal changes in parametrization can induce significant changes to
the class of mfe-structures. We characterized the combinatorial impact
of these changes in terms of the distribution of irreducible substructures
which has practical implications on algorithmic level, i.e.~for the effect
of sparsification of mfe DP-folding of such structures.

We find the following dichotomy: the distribution of irreducible
substructures either is a discrete limit law or a central limit law.
In the former case mfe-structures contain only finitely many irreducible
substructures and the ratio of irreducible mfe-structures over all
mfe-structures becomes in the limit of long sequences a positive
constant. While this means ``just'' a constant time reduction for the
sparsification of the DP-routine, the reduction is typically in the
order of $90$ \% \cite{Fenix12} and consequently still of practical
interest.
In the latter case, the fact that the central limit distribution has a
mean that scales linearly with sequence length alone implies the
these mfe-structures contain a large number of very small irreducible
substructures. As these irreducibles are small it becomes more and more
unlikely to realize a mfe structure as an irreducible. As a result
sparsification has somewhat ``maximal'' effect, i.e.~a linear reduction
in time and space \cite{spar:07}.

From the work of \cite{Han12} we know that a natural RNA structure has
a finite $5'$-$3'$ distance. This means that natural RNA structures
contain only a finite number of irreducible substructures.
In the context of our dichotomy result this means that sparsification of
``realistically'' parameterized mfe structures leads to a constant time
reduction, in accordance with the findings in \cite{Fenix12}.

At the transition point, where the distribution shifts from a discrete to
a central limit law, a local limit law exists. Its density function is that
of a Rayleigh distribution. It is easy to ``test'' our main theorem for
the subcritical and supercritical regimes, see Fig.~\ref{F:distribution},
i.e.~to sample the predicted limit laws. The particular parameters for
the two scenarios displayed in Fig.~\ref{F:distribution} are given in
Tab.~\ref{T:sub-super}. By construction it is practically
impossible to localize the transition and sample the Rayleigh law.

In Fig.~\ref{F:subcritical} we detail two typical RNA secondary structures
sampled from the two regimes. Instead of presenting these structures as
diagrams we map them into trees, where nodes represent irreducible
substructures and edges are being drawn if irreducibles are nested.
The tree representation shows clearly that a small variation of energy
parameters can have a dramatic effect on the mfe structures. It is also
evident why sparsification works much more efficient in the supercritical
regime. Namely, in this case the irreducibles are less complex which implies
that it becomes increasingly unlikely to find any candidates.

\begin{table}
\caption{\small Parameters of (a) (subcritical regime) and (b)
(supercritical regime).}
\label{T:sub-super}      
\begin{tabular}{|l|l|l|l|l|l|l|l|l|}
\hline
 & $\alpha_1$ & $\alpha_2$ & $\alpha_3$ & $\beta_1$ & $\beta_2$ & $\gamma_1$ & $\gamma_2$  \\
\hline
subcritical & $-5$ & $-0.01$ & $7.53$ & $4$ & $-1$ & $-3.4$ & $-0.6$ \\
\hline
supercritical & $-5$ & $-0.01$ & $7.53$ & $2$ & $-1$ & $-10$ & $-3$ \\
\hline
\end{tabular}
\end{table}

\begin{figure}[ht]
\centerline{%
\epsfig{file=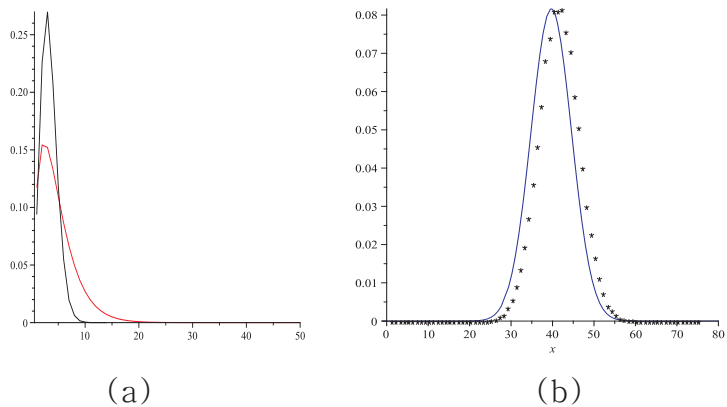,width=0.65\textwidth}\hskip15pt }
\caption{\small The subcritical regime $(a)$: sampling $10^5$ structures
of length $n=700$ (black) versus the discrete limit law (red) as predicted
by Theorem~\ref{T:quotient};
The supercritical regime $(b)$: sampling $10^6$ structures of length $n=10^3$
(black stars) versus the central limit distribution as predicted by
Theorem~\ref{T:quotient} (blue line).
}\label{F:distribution}
\end{figure}

\begin{figure}[ht]
\centerline{%
\epsfig{file=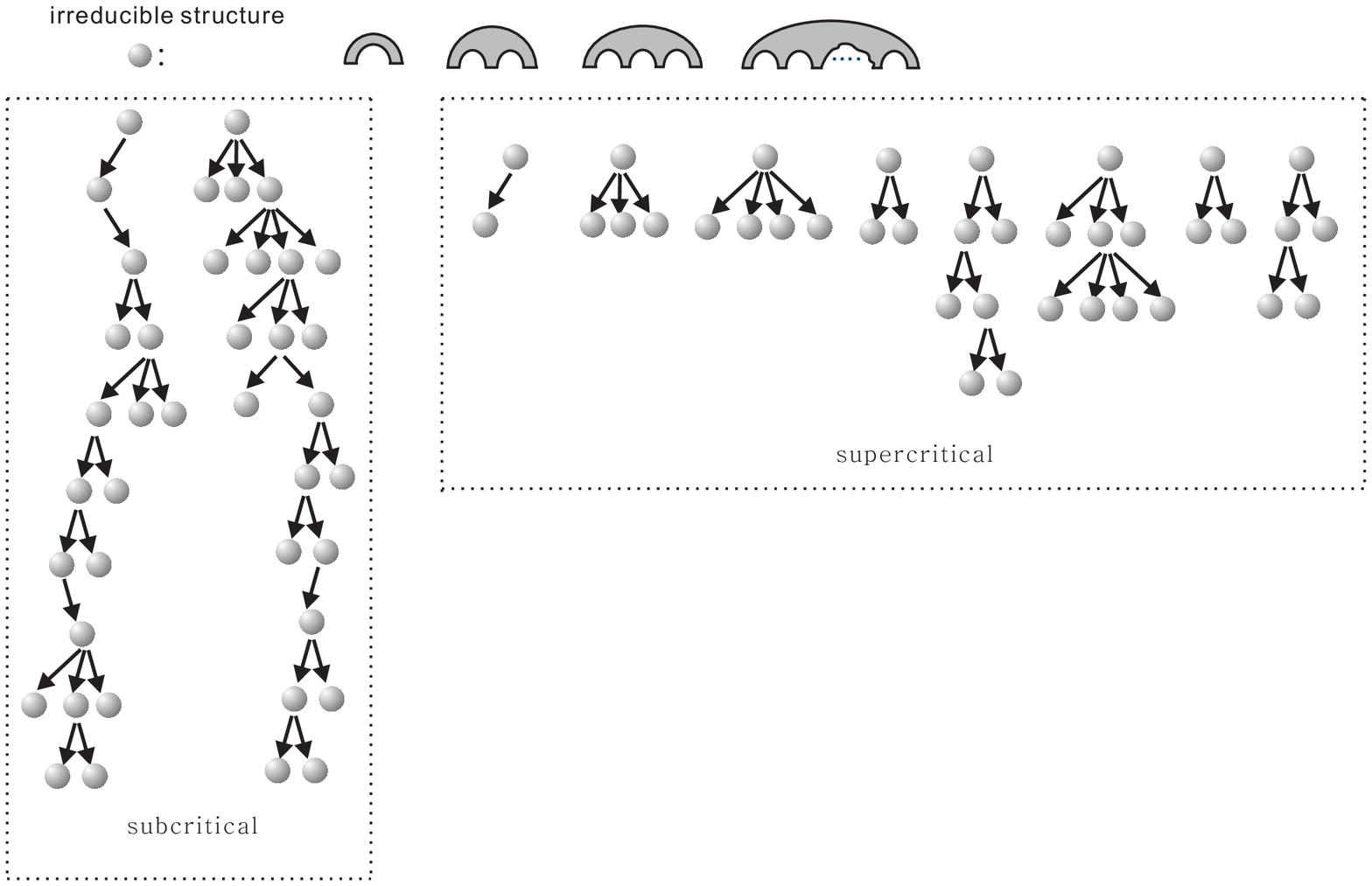,width=0.75\textwidth}\hskip15pt }
\caption{\small Tree representation of structures in the two regimes:
a typical subcritical structure (left) and a typical supercritical
structure (right).
One node in a tree represents an irreducible substructure and
an edge visualizes the nesting relation.
}\label{F:subcritical}
\end{figure}

\section{ Acknowledgments.}
We are grateful to Fenix W.D.~Huang for his help
with the sampling curves of Figure.~\ref{F:distribution}.


\end{document}